\documentclass[a4paper,10pt,reqno]{amsart}
\usepackage{dsfont,amsmath,amsfonts,amscd,amssymb,graphicx,mathrsfs,eufrak}
\usepackage[dvips,all,arc,curve,color,frame]{xypic}
\usepackage[usenames]{color}
\usepackage{setspace}
\usepackage{array,multirow,booktabs,longtable}

\newcommand{\marginparstretch}{0.6}
\let\oldmarginpar\marginpar
\renewcommand\marginpar[1]{\-\oldmarginpar[\framebox{\setstretch{\marginparstretch}\begin{minipage}{\marginparwidth}{\raggedleft\tiny #1}\end{minipage}}]{\framebox{\setstretch{\marginparstretch}\begin{minipage}{\marginparwidth}{\raggedright\tiny #1}\end{minipage}}}}

\usepackage[colorlinks]{hyperref}

\usepackage{tikz,mathrsfs}
\usetikzlibrary{arrows,decorations.pathmorphing,decorations.pathreplacing,positioning,shapes.geometric,shapes.misc,decorations.markings,decorations.fractals,calc,patterns}

\tikzset{
        cvertex/.style={circle,draw=black,inner sep=1pt,outer sep=3pt},
        vertex/.style={circle,fill=black,inner sep=1pt,outer sep=3pt},
        star/.style={circle,fill=yellow,inner sep=0.75pt,outer sep=0.75pt},
        tvertex/.style={inner sep=1pt,font=\scriptsize},
        gap/.style={inner sep=0.5pt,fill=white}}

\tikzstyle{mybox} = [draw=black, fill=blue!10, very thick,
    rectangle, rounded corners, inner sep=10pt, inner ysep=20pt]
\tikzstyle{boxtitle} =[fill=blue!50, text=white,rectangle,rounded corners]

\newtheorem{thm}{Theorem}[section]
\newtheorem{prop}[thm]{Proposition}
\newtheorem{lem}[thm]{Lemma}

\newtheorem{cor}[thm]{Corollary}

\theoremstyle{definition}

\newtheorem{remark}[thm]{Remark}

\numberwithin{equation}{section}

\newcommand{\nc}{\newcommand}
\nc{\R}{\mathbb{R}}
\nc{\Z}{\mathbb{Z}}
\nc{\C}{\mathbb{C}}
\nc{\Q}{\mathbb{Q}}
\nc{\N}{\mathbb{N}}
\nc{\s}{\mathfrak{S}}
\nc{\mc}{\mathcal}
\nc{\mf}{\mathfrak}
\nc{\g}{\mf{g}}
\newcommand{\idot}{{\:\raisebox{2pt}{\text{\circle*{1.5}}}}}
\nc{\CM}{\mathsf{CM}}
\nc{\bc}{\mathbf{c}}
\nc{\uGamma}{{\Gamma'}}
\nc{\ds}{\dots}
\nc{\h}{\mf{h}}
\def\Spec{\mathop{\rm Spec}\nolimits}
\nc{\Cs}{\C^{\times}}
\nc{\an}{\mathrm{an}}
\nc{\aff}{\mathrm{aff}}
\nc{\Yle}{Y_{\le 1}}
\DeclareMathOperator{\Pic}{\mathrm{Pic}}
\nc{\reg}{\mathrm{reg}}
\nc{\Irr}{\mathrm{Irr} }
\DeclareMathOperator{\Hom}{\mathrm{Hom}}
\newcommand{\iso}{{\;\stackrel{_\sim}{\to}\;}}
\nc{\mbf}{\mathbf}
\renewcommand{\o}{\otimes}

\begin{document}

\title{\textsc{Counting resolutions of symplectic quotient singularities}}

\author{Gwyn Bellamy}

\address{School of Mathematics and Statistics, University Gardens, University of Glasgow, Glasgow,  G12 8QW, UK.}
\email{gwyn.bellamy@glasgow.ac.uk}

\begin{abstract}
Let $\Gamma$ be a finite subgroup of $\mathrm{Sp}(V)$. In this article we count the number of symplectic resolutions admitted by the quotient singularity $V / \Gamma$. Our approach is to compare the universal Poisson deformation of the symplectic quotient singularity with the deformation given by the Calogero-Moser space. In this way, we give a simple formula for the number of $\Q$-factorial terminalizations admitted by the symplectic quotient singularity in terms of the dimension of a certain Orlik-Solomon algebra naturally associated to the Calogero-Moser deformation. This dimension is explicitly calculated for all groups $\Gamma$ for which it is known that $V / \Gamma$ admits a symplectic resolution. As a consequence of our results, we confirm a conjecture of Ginzburg and Kaledin.  
\end{abstract}
\thanks{The author would like to thank Y. Namikawa for bring the problem of comparing the universal and Calogero-Moser deformations to his attention, and for several simulating discussions. The author would like to thank T. Schedler and M. Lehn for helpful comments on earlier drafts of the work. He would also like to thank the referee for several insightfully comments. The author is supported by the EPSRC grant~EP-H028153.}
\maketitle
\parindent 20pt
\parskip 0pt


\section{Introduction}

The goal of this article is to count the number of non-isomorphic symplectic resolutions of a symplectic quotient singularity $V / \Gamma$; where $V$ is a finite dimensional complex vector space and $\Gamma \subset \mathrm{Sp}(V)$ a finite group. A \textit{$\Q$-factorial terminalization} of $V / \Gamma$ is a projective, crepant, birational morphism 
$$
\rho : Y \rightarrow V / \Gamma
$$
such that $Y$ has only $\Q$-factorial, terminal singularities. We say that $Y$ is a \textit{symplectic resolution} of $V / \Gamma$ if $Y$ is smooth. It is not always the case that the quotient admits a symplectic resolution, in fact such examples are relatively rare. However, it is a consequence of the minimal model program that $V/\Gamma$ always admits a $\Q$-factorial terminalization. Moreover, work of Namikawa shows that $V/\Gamma$ admits only finitely many $\Q$-factorial terminalizations up to isomorphism, and if one of these $\Q$-factorial terminalizations is actually smooth i.e. is a symplectic resolution, then all $\Q$-factorial terminalizations are smooth. 

The main result of this paper is an explicit formula for the number of $\Q$-factorial terminalizations admitted by $V / \Gamma$. Our approach is to translate the problem into a problem about the singularities of the Calogero-Moser deformation of $V / \Gamma$. Then results about the representation theory of symplectic reflection algebras can be applied to solve the problem. Namely, the centre of the symplectic reflection algebra associated to $\Gamma$ defines a flat Poisson deformation $\CM(\Gamma) \rightarrow \mf{c}$ of $V / \Gamma$. Here the base $\mf{c}$ of the Calogero-Moser deformation is the vector space of class functions supported on the symplectic reflections in $\Gamma$. Let $\mc{Y}$ be a $\Q$-factorial terminalization of $\CM(\Gamma)$ over $\mf{c}$: 
$$
\begin{tikzpicture} 
\node (C1) at (-1,1) {$\mc{Y}$};
\node (C2) at (1,1) {$\CM(\Gamma)$};
\node (C3) at (0,0) {$\mf{c}$};
\draw[->] (C1) -- (C2); 
\draw[->] (C1) -- (C3); 
\draw[->] (C2) -- (C3); 
\node [above] at (0,1) {$\boldsymbol{\rho}$};
\end{tikzpicture}
$$
The set of points $\bc$ for which the map $\boldsymbol{\rho}_{\bc} : \mc{Y}_{\bc} \rightarrow \CM_{\bc}(\Gamma)$ is an isomorphism is denoted $\mf{c}_{\reg}$, and $\mc{D} \subset \mf{c}$ the complement. In \cite{Namikawa2}, Namikawa shows that there is a finite "Weyl group" associated to any affine symplectic variety equipped with a good $\Cs$-action. In particular, we may associate to $V / \Gamma$ its Namikawa Weyl group $W$. Our main result states:

\begin{thm}\label{thm:count}
The number of pairwise non-isomorphic $\Q$-factorial terminalizations  admitted by $V / \Gamma$ equals 
\begin{equation}\label{eq:OSformula}
\frac{1}{|W|} \dim_{\C} H^*(\mf{c} \smallsetminus \mc{D}; \C).
\end{equation}
\end{thm} 

A consequence of our results is that $\mc{D}$ is a union of hyperplanes in $\mf{c}$. This implies that  $H^*(\mf{c} \smallsetminus \mc{D}; \C)$ is the Orlik-Solomon algebra associated to this hyperplane arrangement. Thus, powerful results in algebraic combinatorics can be applied to explicitly calculate the number (\ref{eq:OSformula}) in examples of interest. When $V / \Gamma$ admits a symplectic resolution, $\boldsymbol{\rho}_{\bc}$ is an isomorphism if and only if $\CM_{\bc}(\Gamma)$ is smooth i.e. $\mc{D}$ is precisely the locus of singular fibers. 

There is one infinite series of groups for which it is known that the quotient $V/ \Gamma$ admits a symplectic resolution. These are the wreath product symplectic reflection groups. Let $\Gamma = \s_n \wr G$ acting on $V = \C^{2n}$; where $G$ is a finite subgroup of $SL(2,\C)$. The Weyl group associated to $G$ via the McKay correspondence is denoted $W_G$. The exponents of $W_G$ are denoted $e_1, \ds, e_{\ell}$ and $h$ denotes the Coxeter number of $W_G$. 

\begin{prop}\label{prop:countWeyl}
The number of non-isomorphic symplectic resolutions of $V / \Gamma$ equals 
\begin{equation}\label{eq:beautiful}
\prod_{i = 1}^{\ell} \frac{((n-1) h + e_i + 1)}{e_i + 1}
\end{equation}
\end{prop}

Formula (\ref{eq:beautiful}) plays an important role in the theory of generalized Catalan combinatorics associated to Weyl groups. 

In addition to the above infinite series, it is known that there are two exceptional groups  that admit symplectic resolutions. These are $Q_8 \times_{\Z_2} D_8$ and $G_4$, both acting on a four-dimensional symplectic vector space; it seems likely that these make up all groups admitting symplectic resolutions \cite{SchUn}. In the case of $G_4$, Lehn and Sorger explicitly constructed a pair of non-isomorphic symplectic resolutions of $V / \Gamma$. Our results show that these are the only symplectic resolutions of this quotient. In the case of $Q_8 \times_{\Z_2} D_8$, a computer calculation shows that $\dim_{\C} H^*(\mf{c} \smallsetminus \mc{D}; \C) = 2592$, implying that the quotient singularity admits $81$ distinct symplectic resolutions. Recently, these $81$ symplectic resolutions have been explicitly constructed by M. Donten-Bury and J. A. Wi\'sniewski \cite{Wisniewski}. They also show that these $81$ resolutions are all possible resolution up to isomorphism.  

\subsection{Universal vs. Calogero-Moser deformations}\label{sec:uvsp}

The key to proving Theorem \ref{thm:count} is to make a precise comparison between the formally universal Poisson deformation $\mc{X}$ of $V / \Gamma$ and the Calogero-Moser deformation $\CM(\Gamma)$. As noted above, the base of the Calogero-Moser space is the space $\mf{c}$ of class functions on $\Gamma$ supported on the subset of symplectic reflections. On the other hand, Namikawa has shown that the base of the universal deformation $\mc{X}$ is $H^2(Y;\C) / W$. Thus, there exists a morphism $\mf{c} \rightarrow H^2(Y;\C) / W$ such that 
$$
\CM(\Gamma) \simeq \mc{X} \times_{H^2(Y;\C) / W} \mf{c}.
$$
Our main result, Theorem \ref{thm:main}, is an explicit description of the morphism $\mf{c} \rightarrow H^2(Y;\C) / W$. In order to precisely state our results, we introduce some additional notation. 

A subgroup $\uGamma$ of $\Gamma$ is \textit{parabolic} if it is the stabilizer of some vector $v \in V$. The rank of $\uGamma$ is defined to be $\frac{1}{2} \left(\dim V - \dim V^{\uGamma} \right)$, and we say that $\uGamma$ is \textit{minimal} if it has rank one. In this case $\uGamma$ is isomorphic to a finite subgroup of $SL(2,\C)$. The set of $\Gamma$-conjugacy classes of minimal parabolic subgroups is denoted $\mc{B}$. The variety $V / \Gamma$ is stratified by finitely many symplectic leaves and those leaves $\mc{L}$ whose dimension is $\dim V - 2$ are naturally labeled by the elements of $\mc{B}$. For each $B \in \mc{B}$, we fix a representative $\uGamma$ in $B$ and write $\widetilde{\Xi}(B)$ for the normalizer of $\uGamma$ in $\Gamma$. The quotient $\widetilde{\Xi}(B) / \uGamma$ is denoted $\Xi(B)$. Via the McKay correspondence, there is associated to $\uGamma \subset SL(2,\C)$ a Weyl group $(W(B),\h_B)$, of simply laced type. As explained in section \ref{sec:McKay}, there is a natural linear action of $\Xi(B)$ on $\h_B$. We fix $\mf{a}_B := (\h^*_B)^{\Xi(B)}$. The centralizer $W_B$ of $\Xi(B)$ in $W(B)$ acts on $\mf{a}_B$. We fix a $\Q$-factorial terminalization $\rho : Y \rightarrow V / \Gamma$ of $V / \Gamma$.

\begin{thm}\label{prop:Weylgroup}
The Namikawa Weyl group associated to $V/\Gamma$ is $W := \prod_{B \in \mc{B}} W_B$ acting on 
$$
H^2(Y;\C) \simeq \prod_{B \in \mc{B}} \mf{a}_B.
$$ 
\end{thm}

As noted above, the Calogero-Moser deformation plays a key role in our results. Associated to the pair $(V,\Gamma)$ is the \textit{symplectic reflection algebra} $\mathbf{H}(\Gamma)$ at $t = 0$, as introduced by Etingof and Ginzburg \cite{EG}. This is a non-commutative $\C[\mf{c}]$-algebra, free over $\C[\mf{c}]$, such that the quotient $\mathbf{H}(\Gamma) / \langle \C[\mf{c}]_+ \rangle$ is isomorphic to the skew-group algebra $\C[V] \rtimes \Gamma$. Let $e$ denote the trivial idempotent in $\C \Gamma$, so that $e(\C[V] \rtimes \Gamma)e \simeq \C[V]^{\Gamma}$. The algebra $e \mathbf{H}(\Gamma) e$ is a commutative Poisson algebra, again free over $\C[\mf{c}]$, such that 
$$
e \mathbf{H}(\Gamma) e / \langle \C[\mf{c}]_+ \rangle \simeq \C[V]^{\Gamma},
$$
as Poisson algebras. Thus, $\vartheta : \CM(\Gamma) := \Spec e \mathbf{H} e \rightarrow \mf{c}$ is a flat Poisson deformation of $V / \Gamma$. We call $\CM(\Gamma)$ the \textit{Calogero-Moser deformation} of $V / \Gamma$. The key result at the heart of this paper is the following theorem, which makes explicit the relation between the deformations $\mc{X}$ and $\CM(\Gamma)$ of $V / \Gamma$. 

\begin{thm}\label{thm:main}
The McKay correspondence defines a $W$-equivariant isomorphism $\mf{c} \simeq H^2(Y;\C)$ such that the Calogero-Moser deformation $\CM(\Gamma)  \rightarrow \mf{c}$ is isomorphic to the pull-back along the quotient map $H^2(Y;\C) \rightarrow H^2(Y;\C) / W$ of the formally universal Poisson deformation $\mc{X} \rightarrow H^2(Y;\C) / W$.  
\end{thm} 

In particular, Theorem \ref{thm:main} implies that Conjecture 1.9 of \cite{GK} is true. Results of Namikawa \cite{Namikawa3} on the birational geometry of $Y$ show that the number of $\Q$-factorial terminalizations of $V / \Gamma$ can be computed by counting the number of connected components in the complement to a (finite) real hyperplane arrangement in $H^2(Y;\R)$. Theorem \ref{thm:main} allows us to identify the complexification of this hyperplane arrangement with the set $\mc{D}$. Then Theorem \ref{thm:count} can be deduced from Theorem \ref{thm:main} using standard results from the theory of hyperplane arrangements.  We note an immediate corollary of Theorem \ref{thm:main}. 

\begin{cor}
Let $\Gamma_0$ be the normal subgroup of $\Gamma$ generated by all symplectic reflections. Then the number of $\Q$-factorial terminalizations of $V / \Gamma$ equals the number of $\Q$-factorial terminalizations admitted by $V / \Gamma_0$. 
\end{cor}

Thus, if $\Gamma_0 = \{ 1 \}$, then $V / \Gamma$ is the unique $\Q$-factorial terminalization of $V / \Gamma$. Let $\mc{Y}$ be the formally universal Poisson deformation of the terminialization $Y$. Then $\mc{Y}$ is projective over its affinization $\mc{Y}^{\mathrm{aff}} := \Spec \Gamma(\mc{Y},\mc{O})$. 

\begin{cor}\label{cor:affinization}
 Then there exists an isomorphism of Poisson $H^2(Y;\C)$-schemes $\mc{Y}^{\mathrm{aff}} \simeq \CM(\Gamma)$. 
\end{cor}

The birational geometry of $\Q$-factorial terminalizations of $V / \Gamma$ can also be used to deduce results about the Calogero-Moser deformation. Namely, the following is a partial answer to Question 9.8.4 by Bonnaf\'e and Rouquier \cite{BonnafeRouquier}. 

\begin{cor}\label{cor:singCal}
Let $\mc{D}' \subset \mf{c}$ be the locus over which the fibers of the Calogero-Moser deformation  $\CM(\Gamma)$ are singular. Then $\mc{D}'$ is either a finite union of hyperplanes, or the whole of $\mf{c}$. 
\end{cor}

\subsection{Outline of paper}

In section \ref{sec:proof2}, we give a proof of Theorem \ref{prop:Weylgroup}. Section \ref{sec:CMdef} is devoted to the proof of Theorem \ref{thm:main}. The proof of Corollary \ref{cor:singCal} is given in section \ref{sec:BR}. Then, our main result Theorem \ref{thm:count} is proven in section \ref{sec:proof1}. Finally, we consider specific examples in sections \ref{sec:wreath} and \ref{sec:except}, where formula (\ref{eq:beautiful}) of Proposition \ref{prop:countWeyl} is derived.

\begin{remark}
Throughout, the cohomology group $H^i(Y;\C)$ stands for the singular cohomology of underlying reduced variety, equipped with the \textit{analytic} topology. 
\end{remark}   

\section{Namikawa's Weyl group}\label{sec:proof2}

In this section, we describe Namikawa's Weyl group associated to $V / \Gamma$, thus confirming Theorem \ref{prop:Weylgroup}.

\subsection{}\label{sec:McKay} The symplectic leaves in $V / \Gamma$ are labeled by $\Gamma$-conjugacy classes of parabolic subgroups of $\Gamma$. Let $\underline{\Gamma}$ be a parabolic subgroup. Then the leaf $\mc{L}$ labeled by $\underline{\Gamma}$ is the image under $\pi : V \rightarrow V / \Gamma$ of the set $\{ v \in V \ |\ \Gamma_v = \underline{\Gamma} \}$. If $( V / \Gamma)_{\le 1}$ is the open subset consisting of the open symplectic leaf and all leaves $\mc{L}$ of dimension $\dim V - 2$, then we write $\Yle := \rho^{-1}(( V / \Gamma)_{\le 1})$. The open subset $\Yle$ is contained in the smooth locus of $Y$. 

As in section \ref{sec:uvsp}, we fix $B \in \mc{B}$, $\uGamma$ the corresponding minimal parabolic in $\Gamma$ etc. Let $V_0$ denote the complementary $\uGamma$-module to $V^{\uGamma}$ in $V$; $V_0$ is a two-dimensional symplectic subspace. The open subset of $V^{\uGamma}$ consisting of all points whose stabilizer under $\Gamma$ equals $\uGamma$ is denoted $U$. The group $\Xi(B)$ acts freely on $U \times V_0 / \uGamma$ and the quotient map $\pi$ induces a Galois covering $\sigma : U \times V_0 / \uGamma \rightarrow V / \Gamma$ onto its image, with Galois group $\Xi(B)$.

We choose $b \in U$ and set $p = \pi(b)$. Then $\pi(\{ b \} \times V_0) \simeq V_0 / \uGamma$ is a closed subvariety of $V / \Gamma$. Let $Y_B = \rho^{-1}(V_0 / \uGamma)$ in $Y$, so that $Y_B \subset \Yle$ and $\rho : Y_B \rightarrow V_0 / \uGamma$ is a minimal resolution of singularities. Let $F$ be the exceptional locus of this minimal resolution and $\Irr (F)$ the set of exceptional divisors. Recall that $W(B)$ is the Weyl group associated to $\uGamma$. Let $\Delta_B \subset \h_B^*$ be a set of simple roots. The set of isomorphism classes of \textit{non-trivial} irreducible $\uGamma$-modules is denoted $\Irr (\uGamma)$. By \cite[Theorem 2.2 (i)]{GonzalezSprinbergVerdier}, the McKay correspondence is the pair of bijections
\begin{equation}\label{eq:mckaybij}
\Delta_B \iso \Irr (\uGamma) \iso \Irr (F), \quad \alpha \mapsto \rho(\alpha) \mapsto D_{\rho(\alpha)},
\end{equation}
uniquely defined by the condition 
\begin{equation}\label{eq:pairings}
(D_{\rho(\alpha)},D_{\rho(\beta)}) = \dim \Hom_{\uGamma}( V_0 \o \rho(\alpha), \rho(\beta)) =  - \langle \alpha, \beta \rangle,
\end{equation}
where $( - , - )$ is the intersection pairing and $\langle - , - \rangle$ the Killing form. There is a natural representation theoretic action of $\Xi(B)$ on the set $\Irr (\uGamma) $. For $\lambda \in \Irr (\uGamma)$ and $x \in \Xi(B)$, we have $x \cdot \lambda = {}^x \lambda$, where ${}^x \lambda$ is the $\uGamma$-module, which as a vector space equals $\lambda$, with action $g \cdot v = xgx^{-1} v$ for all $v \in \lambda$. The identity (\ref{eq:pairings}) implies that the induced action of $\Xi(B)$ on $\Delta_B$ is via Dynkin diagram automorphism. 

\subsection{} The group $\Xi(B)$ also acts naturally on $H^2(Y_B; \C)$ as follows. Since the decomposition $V = V^{\uGamma} \oplus V_0$ is as $\widetilde{\Xi}(B)$-modules, $\Xi(B)$ acts on $V_0 / \uGamma \subset V / \uGamma$. There is a unique lift of this action to the resolution $Y_B$, as can be seen from the explicit construction of $Y_B$ as $\mathrm{Hilb}^{\uGamma}(\C^2)$, the dominant component of $\uGamma$-$\mathrm{Hilb}(\C^2)$; see \cite{CBKleinian}. Thus, there is an induced action of $\Xi(B)$ on $H^2(Y_B;\C)$. 

Recall that each divisor $D \in \Irr (F)$ is a rational curve with self-intersection $-2$. For $D \in \Irr (F)$, let $L_D$ denote the corresponding line bundle on $Y_B$ such that $L_D |_D \simeq \mc{O}_D(-1)$ and $L_D |_{D'} = \mc{O}_{D'}$ for $D' \neq D$. The following is a well-known part of the McKay correspondence, but we sketch a proof since we were unable to find a suitable reference. 

\begin{lem}\label{lem:twodimHtwo}
For $\mc{L} \in \Pic(Y_B)$, let $c_1(\mc{L})$ denote its Chern character in $H^2(Y_B;\C)$. 
\begin{enumerate}
\item The Chern characters $c_1(L_D)$, for $D \in \Irr (F)$, are a basis of $H^2(Y_B;\C)$. 
\item The induced isomorphism 
$$
\h^*_B \iso H^2(Y_B;\C), \quad \alpha \mapsto c_1 \left( L_{D_{\rho(\alpha)}} \right)
$$
is $\Xi(B)$-equivariant. 
\end{enumerate}
\end{lem}

\begin{proof}
Both statements will be proven simultaneously. We have defined the action of $\Xi(B)$ on $\Delta_B$ such that the bijection $\Delta_B \iso \Irr (\uGamma)$ is equivariant. Since the action of $\Xi(B)$ on $V_0 / \uGamma$ fixes the singular point, $\Xi(B)$ acts on $F$, permuting its irreducible components. Thus, there is a geometric action of $\Xi(B)$ on $\Irr (F)$. It follows from the beautiful interpretation of the bijection $\Irr (\uGamma) \iso \Irr (F)$ given in \cite{CBKleinian} that this bijection is $\Xi(B)$-equivariant; see also \cite[Section 6.2]{SchUn}. Thus, it suffices to check that the Chern characters $c_1(L_D)$, for $D \in \Irr (F)$, are a basis of $H^2(Y_B;\C)$ such that $x \cdot c_1(L_D) = c_1(L_{x \cdot D})$ for all $x \in \Xi(B)$. 

The action of $\Xi(B)$ on $Y_B$ commutes with the natural conic $\Cs$-action. Therefore, \cite[Proposition 4.3.1]{SlodowyFourLectures} shows that the embedding $F \hookrightarrow Y_B$ induces, by restriction, a $\Xi(B)$-equivariant isomorphism $H^2(Y_B,\C) \iso H^2(F,\C)$. Now, by the Mayer-Viratoris long exact sequence, the embeddings $D \hookrightarrow F$ identify $H^2(F;\C)$ with $\bigoplus_{D \in \Irr (F)} H^2(D;\C)$. Under the identification 
$$
H^2(Y_B,\C)  \iso \bigoplus_{D \in \Irr (F)} H^2(D;\C)
$$
the group $\Xi(B)$ acts by permuting the (one-dimensional) summands of the right-hand side. On the other hand, the image of $c_1(L_D)$ in $H^2(D';\C)$ is either a basis element if $D = D'$, or zero if $D \neq D'$, since $c_1 \left( \mc{O}_{\mathbb{P}^1} \right) = 0$ in $H^2(\mathbb{P}^1;\C)$. The claims of the lemma follow. 
\end{proof}

Define $Z$ to be the fiber product
$$
\begin{tikzpicture} 
\node (C1) at (-1.5,1.5) {$Z$};
\node (C2) at (1.5,1.5) {$Y$};
\node (C3) at (-1.5,0) {$U \times V_0 / \uGamma$};
\node (C4) at (1.5,0) {$V / \Gamma$};
\draw[->] (C1) -- (C2); 
\draw[->] (C1) -- (C3); 
\draw[->] (C2) -- (C4); 
\draw[->] (C3) -- (C4);
\node [above] at (0,1.5) {$\sigma'$};
\node [above] at (0,0) {$\sigma$};
\end{tikzpicture}
$$
Since $\sigma$ is \'etale, $\sigma'$ is also \'etale by base change. The following is based on \cite[Proposition 5.2]{KaledinDynkinArxiv}. 

\begin{prop}\label{prop:Ziso}
There is a $\Xi(B)$-equivariant isomorphism $U \times Y_B \iso Z$. 
\end{prop}

\begin{proof}
Set $U_0 = U \times V_0 / \uGamma$. Since $\sigma$ is \'etale, and $U_0 \times_{V/\Gamma} Y = U_0 \times_{V/\Gamma} \Yle$, the fiber product $Z$ is a smooth variety. Projective base change implies that it is projective over $U_0$. If $(V_0 / \uGamma)_{\reg}$ is the smooth locus of $V_0 / \uGamma$ then $U \times (V_0 / \uGamma)_{\reg}$ is an open subset of $V^{\uGamma} \times (V_0 / \uGamma)_{\reg}$ with compliment of codimension at least two. Hence $\Pic(U \times (V_0 / \uGamma)_{\reg}) \simeq \Pic((V_0 / \uGamma)_{\reg})$ is torsion. Therefore, the proof of \cite[Lemma 5.1]{KaledinDynkinArxiv} implies that there is a sheaf of ideals $\mc{E} \subset \mc{O}_{U_0}$ and an isomorphism 
$$
Z \iso \mathrm{Bl}(U_0, \mc{E}),
$$
of varieties projective over $U_0$, where $\mathrm{Bl}(U_0, \mc{E})$ is the blowup of $U_0$ along $\mc{E}$. Since the line bundle on $Z$, ample relative to $U_0$, used to embed $Z$ in $\mathbb{P}^N_{U_0}$ is the pullback of a line bundle on $Y$, ample relative to $V / \Gamma$, the identification $Z \simeq \mathrm{Bl}(U_0, \mc{E})$ is $\Xi(B)$-equivariant i.e. $\mc{E}$ is $\Xi(B)$-stable. To show that $ \mathrm{Bl}(U_0, \mc{E}) \simeq U \times Y_B$, we follow the proof of \cite[Proposition 5.2]{KaledinDynkinArxiv}. Based on the argument given there, it is clear that it suffices to show that all the vector fields $\mf{t}_v$ on $U_0$ coming from the constant coefficient vector fields $v \in V^{\uGamma}$ admit lifts to $Z$. 

The projective morphism $\rho : \Yle \rightarrow \rho(\Yle)$ is semi-small since $\Yle$ is a symplectic manifold \cite[Theorem 3.2]{FuSurvey}. Therefore, since the map $\sigma : U_0 \rightarrow V / \Gamma$ is finite onto its image, the map $Z \rightarrow U_0$ is also semi-small. Moreover, by \'etale base change, the fact that the canonical bundle on $\Yle$ is trivial implies that the canonical bundle on $Z$ is trivial too. Therefore, the required lifting follows from \cite[Lemma 5.3]{GK}. 
\end{proof}

\begin{lem}\label{lem:fundamental}
The fundamental group $\pi_1(\mc{L})$ of $\mc{L}$ equals $\Xi$. 
\end{lem}

\begin{proof}
The leaf $\mc{L}$ is the image under $\sigma$ of $U \times \{ 0 \} \subset U \times V_0 / \uGamma$. Thus, $\mc{L} \simeq U / \Xi$. Since $\Xi$ acts freely on $U$ this implies that we have a short exact sequence $1 \rightarrow \pi_1(U) \rightarrow \pi_1(\mc{L}) \rightarrow \Xi \rightarrow 1$. Hence, it suffices to show that $\pi_1(U)$ is trivial. The complement of $U$ in $V^{\uGamma}$ is the union of subspaces $V^{\uGamma} \cap V^{\uGamma'}$, where $\uGamma'$ is a parabolic subgroup of $\Gamma$ such that $V^{\uGamma} \cap V^{\uGamma'}$ is a proper subspace of $ V^{\uGamma}$. We may assume that $\uGamma \subsetneq \uGamma'$ so that $V^{\uGamma'} \subsetneq V^{\uGamma}$. But $V^{\uGamma'} $ is a symplectic subspace of $V$. Thus, $\dim V^{\uGamma'} < \dim V^{\uGamma} - 1$. Hence, the compliment of $U$ in $V^{\uGamma}$ has codimension at least two, implying that $\pi_1(U)$ is trivial. 
\end{proof}

If $(V / \Gamma)_0$ is the open leaf in $V / \Gamma$, then we denote by $Y_0$ the preimage of $(V / \Gamma)_0$ under $\rho$. The map $\rho$ is an isomorphism over $Y_0$. 

\begin{lem}\label{lem:Hvanish}
For $0 < i < 4$, the cohomology groups $H^i(U;\C)$ and $H^i(Y_0;\C)$ are zero. 
\end{lem}

\begin{proof}
As shown in the proof of Lemma \ref{lem:fundamental}, the compliment $C$ to $U$ in $V^{\uGamma}$ has complex codimension at least two. Therefore, 
$$
H^{\mathrm{BM}}_j(C;\C) = H^{2 \dim_{\C} V^{\uGamma} - i}\left( V^{\uGamma}_{\R},V^{\uGamma}_{\R} \smallsetminus C; \C \right) = 0, \quad \forall \ j > 2 \dim_{\C} C,
$$
where $\mathrm{BM}$ indicates Borel-Moore homology. This implies that $H^{i}\left(V^{\uGamma}_{\R},V^{\uGamma}_{\R} \smallsetminus C; \C\right) = 0$ for $i < 4$. Since $H^i\left(V^{\uGamma}_{\R};\C\right) = 0$ for $i > 0$, the first claim follows from the long exact sequence in relative cohomology. For the second claim, we note first that if $V_{\reg}$ is the open subset of $V$ on which $\Gamma$ acts freely, then $\rho$ restricts to an isomorphism $Y_0 \iso \pi(V_{\reg})$. On $V_{\reg}$, the map $\pi$ is a covering with Galois group $\Gamma$. Therefore, by \cite[Proposition 3.G.1]{Hatcher}, it suffices to show that $H^i(V_{\reg};\C) = 0$ for $0 < i < 4$. Again, this follows from the fact that the compliment to $V_{\reg}$ in $V$ has complex codimension at least $2$. 
\end{proof}

\begin{lem}\label{lem:H2facts}
Fix a $\Q$-factorial terminalization $\rho : Y \rightarrow V/\Gamma$. Then $H^2(Y_{\le 1};\R) = H^2(Y;\R)$.
\end{lem}

\begin{proof}
Let $Y^o$ denote the smooth locus of $Y$ and $X^o$ its image under $\rho$. Since $Y$ has terminal singularities, the compliment of $Y^o$ in $Y$ has codimension at least $4$ \cite{NamikawaNote}. Then \cite[Theorem 3.2]{FuSurvey} says that $\rho$ restricted to $Y^o$ is a semi-small map. Therefore $Y \smallsetminus \Yle$ has codimension at least $4$ in $Y$. The lemma follows from the argument given in the proof of Lemma \ref{lem:Hvanish} above.
\end{proof}

\begin{prop}\label{prop:H2iso}
The restriction maps $H^2(Y; \C) \rightarrow H^2(Y_B;\C)$ induce an isomorphism 
$$
H^2(Y; \C) \iso \bigoplus_{B \in \mc{B}} H^2(Y_B;\C)^{\Xi(B)}.
$$
\end{prop}

\begin{proof}
By Lemma \ref{lem:H2facts}, it suffices to show that the restriction maps $H^2(\Yle; \C) \rightarrow H^2(Y_B;\C)$ induce an isomorphism 
$$
H^2(\Yle; \C) \iso \bigoplus_{B \in \mc{B}} H^2(Y_B;\C)^{\Xi(B)}.
$$
Let $Y(B) \subset Y$ be the open set $\rho^{-1}(\sigma(U \times V_0 / \uGamma))$. Then for $B \neq B'$ in $\mc{B}$, we have $Y(B) \cap Y(B') = Y_0$ and $\Yle = \bigcup_{B} Y(B)$. We claim that restriction defines an isomorphism 
$$
H^2(\Yle;\C) \iso \bigoplus_{B \in \mc{B}} H^2(Y(B);\C). 
$$
This follows from the Mayer-Vietoris sequence by induction on $|\mc{B}|$, using the fact that $H^i(Y_0; \C) = 0$ for $0 < i < 4$ by Lemma \ref{lem:Hvanish}. Therefore we are reduced to showing that restriction $ H^2(Y(B);\C) \rightarrow H^2(Y_B;\C)$ is injective with image $H^2(Y_B;\C)^{\Xi(B)}$. 

Recall that we identified $Y_B$ with a closed subset of $Y$ by first fixing $b \in U$ and identifying $V_0 / \uGamma$ with $\sigma(\{ b \} \times V_0 / \uGamma)$ in $V / \Gamma$. Therefore, the closed embedding $Y_B \hookrightarrow Y$ factors as $Y_B \stackrel{j}{\rightarrow} Z \stackrel{\sigma}{\rightarrow} Y$, where $j$ is the closed embedding $u \mapsto (b,u)$ in $Z \simeq U \times Y_B$ of Proposition \ref{prop:Ziso}. Then Lemma \ref{lem:Hvanish} and the Kunneth formula imply that $j$ induces an identification  $H^2(Z;\C) = H^2(Y_B;\C)$.  

The image of $Z$ in $Y$ under the natural map $\sigma' : Z \rightarrow Y$ equals $Y(B)$. Recall that this map is just the quotient map for the free action of $\Xi(B)$ on $Z$. Therefore, by \cite[Proposition 3.G.1]{Hatcher}, pullback along $\sigma'$ is injective with image $H^2(Z;\C)^{\Xi(B)} = H^2(Y_B;\C)^{\Xi(B)}$. 
\end{proof}

Theorem \ref{prop:Weylgroup} is now a direct consequence of Proposition \ref{prop:H2iso} and the proof of Theorem 1.1 of \cite{Namikawa2} given in \textit{loc. cit.} In particular, the identification $W \simeq \prod_{B \in \mc{B}} W_B$ follows from Lemma 1.2 of \textit{loc. cit.}

\begin{remark}
Let $\mc{L}$ be the leaf in $V / \Gamma$ labeled by $B \in \mc{B}$. The restriction of $\rho$ to $\rho^{-1}(\mc{L})$ is (in the analytic topology) a fiber bundle with fiber $F$. Therefore, by Lemma \ref{lem:fundamental}, the action of $\Xi(B)$ on $H^2(Y_B; \C) \simeq H^2(F;\C)$ is the monodromy action of $\pi_1(\mc{L}) = \Xi(B)$. 
\end{remark}

\section{Calogero-Moser deformations}\label{sec:CMdef}

Our approach to the proof of Theorem \ref{thm:main} will be by analogy with the proof of \cite[Theorem 1.1]{Namikawa2}. As in previous sections we fix a $\Q$-factorial terminalization $\rho : Y \rightarrow V / \Gamma$.

\subsection{Formally universal Poisson deformations} Recall from Lemma \ref{lem:H2facts} that the cohomology group $H^2(\Yle;\C)$ equals $H^2(Y;\C)$. By \cite[Theorem 5.5]{Namikawa} and \cite[Theorem 1.1]{Namikawa}, there are flat Poisson deformations $\nu : \mc{X} \rightarrow H^2(Y;\C) / W$, and $\boldsymbol{\nu} : \mc{Y} \rightarrow H^2(Y;\C)$, of $V / \Gamma$ and $Y$ respectively, such that the diagram 
\begin{equation}\label{eq:commdiNam}
\begin{tikzpicture} 
\node (C1) at (-1.5,1.5) {$\mc{Y}$};
\node (C2) at (1.5,1.5) {$\mc{X}$};
\node (C3) at (-1.5,0) {$H^2(Y;\C)$};
\node (C4) at (1.5,0) {$H^2(Y;\C) / W$};
\draw[->] (C1) -- (C2); 
\draw[->] (C1) -- (C3); 
\draw[->] (C2) -- (C4); 
\draw[->] (C3) -- (C4);
\node [left] at (-1.5,0.75) {$\boldsymbol{\nu}$};
\node [right] at (1.5,0.75) {$\nu$};
\end{tikzpicture}
\end{equation}
is commutative. Moreover, the natural conic action of the torus $\Cs$ on $V/\Gamma$ lifts to the flat families $\mc{X} \rightarrow H^2(Y;\C) / W$ and $\mc{Y} \rightarrow H^2(Y;\C)$ in such a way that $\lambda \cdot h = \lambda^2 h$ for all $h \in H^2(Y;\C)^* \subset \C[H^2(Y;\C)]$ and $\lambda \in \Cs$. The maps in (\ref{eq:commdiNam}) are equivariant for this action.  

The flat Poisson deformation $\mc{X} \rightarrow H^2(Y;\C) / W$ is universal in the following sense. If $\mc{X}' \rightarrow T$ is a flat Poisson deformation of $V/\Gamma$ over a local Artinian $\C$-scheme, then there exists a unique morphism $T \rightarrow H^2(Y;\C) / W$ such that $\mc{X}' \simeq \mc{X} \times_{H^2(Y;\C) / W} T$ as Poisson schemes over $T$. Notice that the map $T \rightarrow H^2(Y;\C) / W$ necessarily factors through the completion of $H^2(Y;\C) / W$ at zero. Thus,  $\mc{X} \rightarrow H^2(Y;\C) / W$ is said to be the \textit{formally universal} Poisson deformation of $V/\Gamma$. Similarly, the deformation $\mc{Y} \rightarrow H^2(Y;\C)$ of $Y$ is formally universal.   

As written, formally universal Poisson deformations of $V/\Gamma$ are clearly not unique, since the definition only involves the completion of the base of the deformation at the special fiber. However, the torus $\Cs$ acts naturally on the ring of functions on this formal neighborhood of the special fiber, see \cite[Section 5.4]{Namikawa}, and $H^2(Y;\C) / W$ is unique in the sense that it is the globalization, as explained in section \ref{sec:global} below, of the formal neighborhood.  

\subsection{Symplectic reflection algebras}\label{sec:SRA} The set of \textit{symplectic reflections} $\mc{S}$ in $\Gamma$ is the set of all elements $s$ such that $\mathrm{rk}_V (1 - s) = 2$.  Let $S_1, \ds, S_r$ be the $\Gamma$-conjugacy classes in $\mc{S}$ and $\bc_1,\ds ,\bc_r$ the characteristic functions on $\mc{S}$ such that $\bc_i(s) = 1$ if $s \in S_i$, and is zero otherwise. The linear span of $\bc_1,\ds ,\bc_r$ is denoted $\mf{c}$. Since we do not require the explicit definition of the symplectic reflection algebras $\mathbf{H}(\Gamma)$, and will only use results from \cite{LosevSRAComplete} about them, we refer the read to \textit{loc. cit.} for their definition. Recall that $\uGamma_B$ is a representative in the conjugacy class  $B$ of minimal parabolic subgroups of $\Gamma$. 

\begin{lem}\label{lem:bij1}
The natural map $\zeta : \bigsqcup_{B \in \mc{B}} (\uGamma_B \smallsetminus \{ 1 \}) / \widetilde{\Xi}(B) \longrightarrow \mc{S} / \Gamma$ is a bijection. 
\end{lem}

Let $\mf{c}_B$ be the subspace of $\mf{c}$ spanned by all $\bc_i$ such that $S_i \cap \uGamma_B \neq \emptyset$. Lemma \ref{lem:bij1} implies that $\mf{c} = \bigoplus_{B \in \mc{B}} \mf{c}_B$. Choose $B \in \mc{B}$. Via the McKay correspondence (\ref{eq:mckaybij}), an element $h \in \h^*_B$ can be considered as a linear combination of the non-trivial characters of $\uGamma_B$. In other words, it is a class function on $\uGamma_B$. Hence an element of $\mf{a}_B = (\h_B^*)^{\Xi(B)}$ is a $\widetilde{\Xi}(B)$-equivariant function $\uGamma_B \rightarrow \C$, where $\widetilde{\Xi}(B)$ acts trivially on $\C$. Thus, we may define an isomorphism 
\begin{equation}\label{eq:mcKay}
\varpi : \mf{c} = \bigoplus_{B \in \mc{B}} \mf{c}_B \stackrel{\sim}{\longrightarrow} \bigoplus_{B \in \mc{B}}  \mf{a}_B, \quad \mf{c}_B \ni \bc \ \mapsto \ (g \mapsto \bc(\zeta(g))). 
\end{equation}
As explained in section \ref{sec:uvsp}, the spherical subalgebra $e \mathbf{H}(\Gamma) e$ is a commutative $\C[\mf{c}]$-subalgebra, equipped with a natural Poisson structure, such that the flat family $\vartheta : \CM(\Gamma) \rightarrow \mf{c}$ is a Poisson deformation of $V / \Gamma$.   

\subsection{Globalization}\label{sec:global}
Suppose we have two conic affine varieties $X$ and $Y$ i.e. $\C[X]$ and $\C[Y]$ are positively graded algebras with degree zero part equal to $\C$, and an equivariant morphism $\gamma : X \rightarrow Y$. Let $X^{\wedge}$ and $Y^{\wedge}$ denote the completions of $X$ and $Y$ respectively at the $\Cs$-fixed point. As shown in \cite[Lemma (A.2)]{NamikawaPoissondeformations}, $\C[X]$ is the ring of $\Cs$-locally finite ($=$ rational) vectors in $\C[X^{\wedge}]$. We say that $\gamma$ is the \textit{globalization} of $\hat{\gamma} : X^{\wedge} \rightarrow Y^{\wedge}$ if, under the identification of $\C[X]$ with rational vectors in $\C[X^{\wedge}]$ and similarly for $\C[Y] \subset \C[Y^{\wedge}]$, $\gamma$ is just the restriction of $\hat{\gamma}$; see \cite[Appendix]{NamikawaPoissondeformations}. 

If $\mc{X}^{\wedge}$ is the completion of $\mc{X}$ along the closed subvariety $V / \Gamma$ and $(H^2(Y;\C) / W)^{\wedge}$ the completion of $H^2(Y;\C) / W$ at $o$, then $\nu$ is the globalization of the induced map of formal schemes $\mc{X}^{\wedge} \rightarrow (H^2(Y;\C) / W)^{\wedge}$. The latter is the universal Poisson deformation of $V / \Gamma$ in the category of pro-Artinian local $\C$-algebras. The analogous statement holds for $\mc{Y} \rightarrow H^2(Y;\C)$; see \cite[Section 5]{Namikawa}. 

The Calogero-Moser deformation $\vartheta : \CM(\Gamma) \rightarrow \mf{c}$ of $V / \Gamma$ is also $\Cs$-equivariant, where $\Cs$ acts on $\mf{c}^* \subset \C[\mf{c}]$ by $\lambda \cdot \bc = \lambda^2 \bc$. This is a consequence of the fact that the symplectic reflection algebra $\mathbf{H}(\Gamma)$ is naturally $\N$-graded, such that $\mf{c}^* \subset \mathbf{H}(\Gamma)$ has degree two, $V^*$ has degree one and $\Gamma$ sits in degree zero. Moreover, if $\mathbf{H}(\Gamma)^{\wedge}$ is the completion of $\mathbf{H}(\Gamma)$ along the two-sided ideal generated by $\mf{c}^*$, then one can identify $\mathbf{H}(\Gamma)$ with the subalgebra of $\mathbf{H}(\Gamma)^{\wedge}$ of rational vectors. This implies that $\CM(\Gamma) \rightarrow \mf{c}$ is the globalization of $\CM(\Gamma)^{\wedge} \rightarrow \widehat{\mf{c}}$, where $\CM(\Gamma)^{\wedge}$ is the completion of $\CM(\Gamma)$ along $V / \Gamma$ and $\widehat{\mf{c}}$ the completion of $\mf{c}$ at zero. Hence, there exists a unique $\Cs$-equivariant morphism $\hat{\alpha} : \widehat{\mf{c}} \rightarrow (H^2(Y;\C) / W)^{\wedge}$ such that $\CM(\Gamma)^{\wedge} \simeq \widehat{\mf{c}} \times_{(H^2(Y;\C) / W)^{\wedge}} \mc{X}^{\wedge}$. This implies:

\begin{lem}\label{lem:classmap}
There exists a unique $\Cs$-equivariant map $\alpha : \mf{c} \rightarrow H^2(Y;\C) / W$ such that 
$$
\CM(\Gamma) \simeq \mf{c} \times_{H^2(Y;\C) / W} \mc{X}.
$$
\end{lem}

On the other hand, the linear isomorphism (\ref{eq:mcKay}) together with the quotient map $H^2(Y;\C) \rightarrow H^2(Y;\C) / W$ defines a map $\beta : \mf{c} \rightarrow H^2(Y;\C) / W$ which is clearly also the globalization of $\hat{\beta} : \widehat{\mf{c}} \rightarrow   (H^2(Y;\C) / W)^{\wedge}$. Theorem \ref{thm:main} is claiming that $\alpha = \beta$. It suffices instead to show that 
$$
\hat{\alpha} = \hat{\beta} : \widehat{\mf{c}} \rightarrow   (H^2(Y;\C) / W)^{\wedge}. 
$$
This will be our goal for the remainder of the section.

\subsection{Kleinian singularities}  In this section we consider the case $\dim V = 2$, and hence $\Gamma$ is a Kleinian group. As noted in section \ref{sec:uvsp}, associated to $\Gamma$  via the McKay correspondence is a Weyl group $(W,\h)$. Let $Y$ be the minimal resolution of $V / \Gamma$. As in Lemma \ref{lem:twodimHtwo}, we have a natural identification $\h^* \rightarrow H^2(Y;\C)$. Therefore, the formally universal Poisson deformation is a flat family $\mc{X} \rightarrow \h^* / W$. 

Fix a finite group $\Gamma \subset \widetilde{\Xi} \subset N_{SL(2,\C)}(\Gamma)$. Lemma \ref{lem:twodimHtwo} implies that the quotient $\Xi := \widetilde{\Xi} / \Gamma$ acts on $\h^*$ via Dynkin diagram automorphisms. In this case, $\mf{c}$ is the space of all $\Gamma$-equivariant functions $\Gamma \smallsetminus \{ 1 \} \rightarrow \C$, the action of $\Gamma$ on $\C$ being trivial. The group $\Xi$ acts on $\mf{c}$ by $(x \cdot \chi)(s) = \chi(\tilde{x} s \tilde{x}^{-1})$, where $\tilde{x}$ is some lift of $x$ to $\widetilde{\Xi}$. This action extends uniquely to an action of $ \widetilde{\Xi}$ on $\mathbf{H}(\Gamma)$ by algebra automorphisms such that the restriction of this action to $\Gamma$ is just conjugation. The action preserves the spherical subalgebra $e \mathbf{H}(\Gamma) e$, the action of $ \widetilde{\Xi}$ on this subalgebra factoring through $\Xi$. Thus, $\Xi$ acts on $\CM(\Gamma)$ such that the map $\CM(\Gamma) \rightarrow \mf{c}$ is equivariant. Since the action of $\Xi$ on  $e \mathbf{H}(\Gamma) e$ can be extended to the case where $t = 1$ (or more generally a formal variable $t$), $\Xi$ acts on $\CM(\Gamma)$ via Poisson automorphisms. Recall that we have defined in (\ref{eq:mcKay}) an isomorphism $\varpi : \mf{c} \iso \h^*$; this is an $\Xi$-equivariant isomorphism, where $\Xi$ acts on $\h^*$ as defined in section \ref{sec:McKay}.  

\begin{lem}\label{lem:rank2}
The map $\varpi$ extends to a $\Cs$-equivariant isomorphism $\CM(\Gamma) \simeq \mf{c} \times_{\h^* / W} \mc{X}$.  
\end{lem}

\begin{proof}
Let $\mf{g}$ be the simple Lie algebra associated to $\Gamma$ under the McKay correspondence. It is well-known, e.g. \cite[Proposition 3.1 (1)]{Namikawa}, that a Slodowy slice $S \rightarrow \h^* / W$ to the subregular nilpotent orbit in $\mf{g}$ is the formally universal Poisson deformation of the Kleinian singularity $V / \Gamma$. Let $\widetilde{S} \rightarrow \h^*$ be the resolution of the formally universal deformation $S \rightarrow \h^* / W$ of $V / \Gamma$ coming from taking the preimage of $S$ in Grothendieck's simultaneous resolution of $\g^*$. By \cite[Proposition 6.2]{SlodowyNamLehnSorger}, $\widetilde{S} \rightarrow \h^*$ is the formally universal Poisson deformation of the minimal resolution of $V / \Gamma$.  

Theorems 6.2.2 and 5.3.1 of \cite{Losev} imply that there is an isomorphism $\CM(\Gamma) \iso \widetilde{S}^{\aff}$ such that the following diagram commutes
$$
\begin{tikzpicture} 
\node (C1) at (-1.2,1.2) {$\CM(\Gamma)$};
\node (C2) at (1.2,1.2) {$\widetilde{S}^{\aff}$};
\node (C3) at (-1.2,0) {$\mf{c}$};
\node (C4) at (1.2,0) {$\h^*$};
\draw[->] (C1) -- (C2); 
\draw[->] (C1) -- (C3); 
\draw[->] (C2) -- (C4); 
\draw[->] (C3) -- (C4);
\node [above] at (0,0) {$\varpi$};
\node [above] at (0,1.2) {$\sim$};
\end{tikzpicture}
$$
This implies the statement of the lemma. 
\end{proof}

\subsection{Factorization of the Calogero-Moser space}\label{sec:factor}

Fix $B \in \mc{B}$, $\uGamma$ the corresponding minimal parabolic subgroup and $\mc{L}$ the symplectic leaf in $V / \Gamma$ labeled by $B$. Let $\mf{d} \simeq \h^*_B$ be the base of the Calogero-Moser deformation of $\uGamma$ so that $\mathbf{H}(\uGamma)$ is a $\C[\mf{d}]$-algebra. For clarity, we write $\mathbf{H}_{\mf{d}}(\uGamma) := \mathbf{H}(\uGamma)$ to show the dependence on $\mf{d}$. As explained in section \ref{sec:SRA}, the space $\mf{c}_B = \mf{d}^{\Xi(B)}$ is a subspace and projection followed by inclusion defines a linear map $\mf{c} \rightarrow \mf{d}$. Let $\mathbf{H}_{\mf{c}}(\uGamma) := \C[\mf{c}] \otimes_{\C[\mf{d}]} \mathbf{H}_{\mf{d}}(\uGamma)$ denote the symplectic reflection algebra obtained from $\mathbf{H}(\uGamma)$ by base change from $\mf{d}$ to $\mf{c}$. Choose $p \in \mc{L} \subset V / \Gamma$. We may think of $p$ as a $\Gamma$-orbit in $V$. If $I_p$ is the ideal of functions in $\C[V]$ vanishing on this orbit, then $I_p \rtimes \Gamma$ is a two-sided ideal in $\C[V] \rtimes \Gamma$. Recall that  $\C[V] \rtimes \Gamma$ is the quotient of $\mathbf{H}(\Gamma)$ by the ideal generated by $\C[\mf{c}]_+$. Following Losev, we denote by $\mathbf{H}(\Gamma)^{\wedge p}$ the completion of $\mathbf{H}(\Gamma)$ by the preimage of the $I_p \rtimes \Gamma$ under the quotient map. Since the preimage of $I_p \rtimes \Gamma$ in $\mathbf{H}(\Gamma)$  contains $\C[\mf{c}]_+$, the completion $\mathbf{H}(\Gamma)^{\wedge p}$ is a topological $\C[[\mf{c}]]$-algebra. Similarly, $ \mathbf{H}_{\mf{c}}(\uGamma)^{\wedge 0}$ is the completion of $ \mathbf{H}_{\mf{c}}(\uGamma)$ corresponding to the ideal $\C[V_0]_+ \rtimes \uGamma$ of $\C[V_0] \rtimes \uGamma$. The key result \cite[Theorem 2.5.3]{LosevSRAComplete} says 

\begin{thm}\label{thm:Losev}
There is an isomorphism 
$$
\theta^* : \mathbf{H}(\Gamma)^{\wedge p} \rightarrow \mathrm{Mat}_{|\Gamma / \uGamma|} \left( \mathbf{H}_{\mf{c}}(\uGamma)^{\wedge 0} \widehat{\o} \C[V^{\uGamma}] \right)
$$
of topological $\C[[\mf{c}]]$-algebras. 
\end{thm}

Let $e$ and ${e'}$ denote the trivial idempotents in the group algebras of $\Gamma$ and $\uGamma$ respectively, so that $\mathsf{CM}_{\mf{c}}(\uGamma) = \Spec {e'}  \mathbf{H}_{\mf{c}}(\uGamma) {e'} $ is a Poisson variety over $\mf{c}$. Applying the idempotent $e$ to both sides of the isomorphism $\theta^*$ of Theorem \ref{thm:Losev} gives an isomorphism $e (\mathbf{H}(\Gamma)^{\wedge p}) e \rightarrow {e'} ( \mathbf{H}_{\mf{c}}(\uGamma)^{\wedge 0}) {e'} \ \widehat{\o} \ \C[V^{\uGamma}]$; see \cite[Section 2.3]{LosevSRAComplete}. The isomorphism $\theta^*$ of Theorem \ref{thm:Losev} is actually valid for any $t$. This implies that the isomorphism $e( \mathbf{H}(\Gamma)^{\wedge p}) e \rightarrow {e'} ( \mathbf{H}_{\mf{c}}(\uGamma)^{\wedge 0}) {e'} \ \widehat{\o} \ \C[V^{\uGamma}]$ is an isomorphism of \textit{Poisson} algebras.

\begin{cor}\label{cor:Losev}
There is an isomorphism of formal Poisson schemes
$$
\theta : \CM(\Gamma)^{\wedge p} \rightarrow \mathsf{CM}_{\mf{c}}(\uGamma)^{\wedge 0} \times V^{\uGamma} 
$$
over $\widehat{\mf{c}}$. 
\end{cor}

\subsection{} Recall that Lemma \ref{lem:classmap} says that there is a $\Cs$-equivariant morphism $\alpha : \mf{c} \rightarrow H^2(Y;\C) / W$ such that $\CM(\Gamma) \simeq \mf{c} \times_{H^2(Y;\C) / W}  \mc{X}$. Completing at $0 \in \mf{c}$ and $o \in H^2(Y,\C) / W$, we have a Cartesian square
$$
\begin{tikzpicture} 
\node (C1) at (-1.5,1.5) {$\CM(\Gamma)^{\wedge}$};
\node (C2) at (1.5,1.5) {$\mc{X}^{\wedge}$};
\node (C3) at (-1.5,0) {$\widehat{\mf{c}}$};
\node (C4) at (1.5,0) {$(H^2(Y;\C) / W)^{\wedge}$};
\draw[->] (C1) -- (C2); 
\draw[->] (C1) -- (C3); 
\draw[->] (C2) -- (C4); 
\draw[->] (C3) -- (C4);
\node [above] at (-0.6,0) {$\hat{\alpha}$};
\end{tikzpicture}
$$
such that $\alpha$ is the algebraization of $\hat{\alpha}$. Here $\CM(\Gamma)^{\wedge}$ and $\mc{X}^{\wedge}$ are the completions of $\CM(\Gamma)$ and $\mc{X}$ respectively along the special fiber $V / \Gamma$. Choose some $B \in \mc{B}$ and consider $\uGamma := \Gamma_B$, $\mc{L}_B$ etc. Fix $p \in \mc{L}_B$. Completing $\CM(\Gamma)^{\wedge}$ at $p$, the above diagram becomes the Cartesian square
\begin{equation}\label{eq:carA}
\begin{tikzpicture} 
\node (C1) at (-1.5,1.5) {$\CM(\Gamma)^{\wedge p}$};
\node (C2) at (1.5,1.5) {$\mc{X}^{\wedge p}$};
\node (C3) at (-1.5,0) {$\widehat{\mf{c}}$};
\node (C4) at (1.5,0) {$(H^2(Y;\C) / W)^{\wedge}$};
\draw[->] (C1) -- (C2); 
\draw[->] (C1) -- (C3); 
\draw[->] (C2) -- (C4); 
\draw[->] (C3) -- (C4);
\node [above] at (-0.6,0) {$\hat{\alpha}$};
\end{tikzpicture}
\end{equation}

Recall that Theorem \ref{prop:Weylgroup} says that $H^2(Y;\C) / W$ is isomorphic to $\prod_{B \in \mc{B}} \mf{c}_B / W_B$. The projection map $H^2(Y;\C) / W \rightarrow \mf{c}_B / W_B$, followed by the canonical morphism $\mf{c}_B / W_B \rightarrow \h_B^* / W(B)$ is denoted $r_B$, and its completion at $0$ is $\hat{r}_B$. Let $\mc{X}_0(\uGamma)$ denote the formally universal Poisson deformation of $V_0 / \uGamma$. The completion of $\mc{X}_0(\uGamma)$ at $o \in V_0 / \uGamma \subset \mc{X}_0(\uGamma)$ is denoted $\mc{X}_0(\uGamma)^{\wedge 0}$.

\begin{lem}\label{lem:Q2}
Let $p \in \mc{L}_B$ and $\mc{X}^{\wedge p}$ the completion of $\mc{X}$ at $p$. Then the following commutative diagram 
\begin{equation}\label{eq:carB}
\begin{tikzpicture} 
\node (C1) at (-2,1.5) {$\mc{X}^{\wedge p}$};
\node (C2) at (2,1.5) {$\mc{X}_0(\uGamma)^{\wedge 0} \times (V^{\uGamma})^{\wedge 0}$};
\node (C3) at (-2,0) {$(H^2(Y;\C) / W)^{\wedge}$};
\node (C4) at (2,0) {$(\h_B^* / W(B))^{\wedge}$};
\draw[->] (C1) -- (C2); 
\draw[->] (C1) -- (C3); 
\draw[->] (C2) -- (C4); 
\draw[->] (C3) -- (C4);
\node [above] at (0.2,0) {$\hat{r}_B$};
\end{tikzpicture}
\end{equation}
is Cartesian.
\end{lem}

\begin{proof}
The analytic germ of $0$ in $H^2(Y,\C) / W$, resp. in $\h_B^* / W(B)$, is denoted $\mathrm{PDef}(V / \Gamma)$, resp. $\mathrm{PDef}(V / \uGamma)$. They are the Poisson Kuranishi spaces of the corresponding analytic symplectic varieties. Passing to the analytic topology, the formally universal deformation $\mc{X} \rightarrow H^2(Y;\C) / W$ induces a flat Poisson deformation $\mc{X}^{\an} \rightarrow (H^2(Y;\C) / W)^{\an}$. Restricting to the germ of $o$ in $H^2(Y;\C) / W$, we have a flat Poisson deformation $\mc{X}^{\an} \rightarrow \mathrm{PDef}(V / \Gamma)$. 

Passing to the germ of $p \in (V / \Gamma)^{\an} \subset \mc{X}^{\an}$ gives a flat family $(\mc{X}^{\an}, p) \rightarrow \mathrm{PDef}(V / \Gamma)$. This is a deformation of $((V / \Gamma)^{\an}, p)$. By the generalized Darboux Theorem, \cite[Lemma 1.3]{Namikawa}, we have an isomorphism of symplectic varieties 
$$
((V / \Gamma)^{\an}, p) \simeq ((V_0 / \uGamma \times V^{\uGamma})^{\an}, 0).
$$
Moreover, by \cite[Proposition 3.1]{Namikawa}, the universal Poisson deformation of $((V_0 / \uGamma)^{\an}, 0) \times ( (V^{\uGamma})^{\an}, 0)$ is $((\mc{X}_0(\uGamma) \times V^{\uGamma})^{\an}, 0) \rightarrow \mathrm{PDef}(V_0 / \uGamma)$. Hence there exists a holomorphic map $\phi_B : \mathrm{PDef}(V / \Gamma) \rightarrow \mathrm{PDef}(V_0 / \uGamma)$ such that the following diagram is Cartesian
$$
\begin{tikzpicture} 
\node (C1) at (-2,1.5) {$(\mc{X}^{\an}, p)$};
\node (C2) at (2,1.5) {$((\mc{X}_0(\uGamma) \times V^{\uGamma})^{\an}, 0)$};
\node (C3) at (-2,0) {$\mathrm{PDef}(V / \Gamma)$};
\node (C4) at (2,0) {$ \mathrm{PDef}(V_0 / \uGamma)$};
\draw[->] (C1) -- (C2); 
\draw[->] (C1) -- (C3); 
\draw[->] (C2) -- (C4); 
\draw[->] (C3) -- (C4);
\node [above] at (0,0) {$\phi_B$};
\end{tikzpicture}
$$
The map $\phi_B$ is precisely the map constructed in section 4 (i) of the proof of \cite[Theorem 1.1]{Namikawa2}. As explained in section 4 of the proof of of \cite[Theorem 1.1]{Namikawa2}, the completion of $\phi_B$ equals $\hat{r}_B$. Passing to the formal neighborhood of $p$ in $(\mc{X}^{\an}, p)$ and $0$ in $((\mc{X}_0(\uGamma) \times V^{\uGamma})^{\an}, 0)$, we get the Cartesian square stated in the lemma.   
\end{proof}

\subsection{The proof of Theorem \ref{thm:main}} By Corollary \ref{cor:Losev}, we have an isomorphism of Poisson varieties $\CM(\Gamma)^{\wedge p} \simeq \mathsf{CM}_{\mf{c}}(\uGamma)^{\wedge 0} \times V^{\uGamma} $ over $\widehat{\mf{c}}$. Under the identification (\ref{eq:mcKay}), we have a natural decomposition $\mf{c} = \bigoplus_{B \in \mc{B}} \mf{c}_B$. Therefore, if we write $q_B$ for projection from $\mf{c}$ onto $\mf{c}_B$, the square 
\begin{equation}\label{eq:carC}
\begin{tikzpicture} 
\node (C1) at (-1.5,1.5) {$\mathsf{CM}_{\mf{c}}(\uGamma)^{\wedge 0} \times V^{\uGamma}$};
\node (C2) at (1.5,1.5) {$\CM(\Gamma)^{\wedge p}$};
\node (C3) at (-1.5,0) {$\widehat{\mf{c}}_B$};
\node (C4) at (1.5,0) {$\widehat{\mf{c}}$};
\draw[->] (C2) -- (C1); 
\draw[->] (C1) -- (C3); 
\draw[->] (C2) -- (C4); 
\draw[->] (C4) -- (C3);
\node [above] at (0,0) {$\hat{q}_B$};
\end{tikzpicture}
\end{equation}
is also (trivially) Cartesian. Recall from section \ref{sec:factor} that $\mf{d}$ is the natural parameter space associated to the symplectic reflection algebra $\mathbf{H}_{\mf{d}}(\uGamma)$ and $\mf{c}_B = \mf{d}^{\Xi(B)}$. Let $\alpha_B$ be the composite $\mf{c}_B \hookrightarrow \mf{d} \iso \h^*_B \rightarrow \h^*_B / W(B)$ and $\hat{\alpha}_B$ will be its completion at $0$. Lemma \ref{lem:rank2} implies that the following diagram is Cartesian 
\begin{equation}\label{eq:carD}
\begin{tikzpicture} 
\node (C1) at (-1.5,1.5) {$\mc{X}_0(\uGamma)^{\wedge 0} \times V^{\uGamma}$};
\node (C2) at (1.5,1.5) {$\CM_{\mf{c}}(\uGamma)^{\wedge 0} \times V^{\uGamma}$};
\node (C3) at (-1.5,0) {$(\h_B / W(B))^{\wedge}$};
\node (C4) at (1.5,0) {$ \widehat{\mf{c}}_B$};
\draw[->] (C2) -- (C1); 
\draw[->] (C1) -- (C3); 
\draw[->] (C2) -- (C4); 
\draw[->] (C4) -- (C3);
\node [above] at (0.5,0) {$\hat{\alpha}_B$};
\end{tikzpicture}
\end{equation}
The composite of the two bottom horizontal arrows is denoted $\hat{\alpha}_B$. By Lemma \ref{lem:rank2}, $\hat{\alpha}_B = \hat{\beta}_B$. Combining diagrams (\ref{eq:carA}), (\ref{eq:carB}), (\ref{eq:carC}) and (\ref{eq:carD}), we get the following diagram, where each square is Cartesian.
$$
\begin{tikzpicture} 
\node (C1) at (-6,1.5) {$\mc{X}_0(\uGamma)^{\wedge} \times V^{\uGamma}$};
\node (C2) at (-3,1.5) {$\CM_{\mf{c}}(\uGamma)^{\wedge 0} \times V^{\uGamma}$};
\node (C3) at (0,1.5) {$\CM(\Gamma)^{\wedge p}$};
\node (C4) at (2.5,1.5) {$\mc{X}^{\wedge p}$};
\node (C5) at (5.7,1.5) {$\mc{X}_0(\uGamma)^{\wedge 0} \times V^{\uGamma}$};
\node (C6) at (-6,0) {$(\h_B / W(B))^{\wedge}$};
\node (C7) at (-3,0) {$\widehat{\mf{c}}_B$};
\node (C8) at (0,0) {$\widehat{\mf{c}}$};
\node (C9) at (2.5,0) {$(H^2(Y,\C) / W)^{\wedge}$};
\node (C10) at (5.7,0) {$(\h_B^* / W(B) )^{\wedge}$};
\draw[->] (C2) -- (C1); 
\draw[->] (C3) -- (C2); 
\draw[->] (C3) -- (C4); 
\draw[->] (C4) -- (C5);
\draw[->] (C1) -- (C6);  
\draw[->] (C2) -- (C7);
\draw[->] (C3) -- (C8); 
\draw[->] (C4) -- (C9);  
\draw[->] (C5) -- (C10);
\draw[->] (C7) -- (C6);
\draw[->] (C8) -- (C7);  
\draw[->] (C8) -- (C9);
\draw[->] (C9) -- (C10);
\node [above] at (-4,0) {$\hat{\alpha}_B$};
\node [above] at (-1.4,0) {$\hat{q}_B$};
\node [above] at (0.7,0) {$\hat{\alpha}$};
\node [above] at (4.25,0) {$\hat{r}_B$};
\end{tikzpicture}
$$
The universality of the formal Poisson deformation $\mc{X}_0(\uGamma)^{\wedge} \times V^{\uGamma} \rightarrow (\h_B^* / W(B))^{\wedge}$ of $V / \uGamma = V_0/ \uGamma \times V^{\uGamma}$ implies that $\hat{\alpha}_B  \circ \hat{q}_B = \hat{p}_B \circ \hat{\alpha}$, c.f. \cite[Section 1.3]{GK}. Hence
$$
\bigoplus_{B \in \mc{B}} \hat{\alpha}_B  \circ \hat{q}_B = \bigoplus_{B \in \mc{B}} \hat{r}_B \circ \hat{\alpha} = \hat{\alpha} 
$$
since $\bigoplus_{B \in \mc{B}} \hat{r}_B = \mathrm{id}$. On the other hand, it is clear from the explicit definition of $\hat{\beta}$ that $\hat{\beta} = \bigoplus_{B \in \mc{B}}  \hat{\alpha}_B  \circ \hat{q}_B$. This completes the proof of Theorem \ref{thm:main}.

\subsection{}\label{sec:BR}

We turn to the proof of Corollary \ref{cor:singCal}. If the quotient $V / \Gamma$ admits a smooth projective, symplectic resolution then by Theorem \ref{thm:main} and the main theorem of \cite{Namikawa3}, the set of points in $\mf{c}$ for which $\CM_{\bc}(\Gamma)$ is singular is a union of hyperplanes. If $V / \Gamma$ does not admit a symplectic resolution then, by \cite[Corollary 1.21]{GK},  the space $\CM_{\bc}(\Gamma)$ is always singular. 

\section{Counting resolutions}

In this section we deduce Theorem \ref{thm:count} from Theorem \ref{thm:main}, using the main theorem from \cite{Namikawa3}. 

\subsection{}\label{sec:proof1} We recall the main result from \cite{Namikawa3}. As explained in \textit{loc. cit.}, the birational geometry of $Y$ is controlled by the real space $H^2(Y;\R)$. In particular, a key role is played by the movable cone $\mathrm{Mov}(\rho) \subset H^2(Y;\R)$ of $\rho : Y \rightarrow V / \Gamma$; see \textit{loc. cit.} for the definition. The ample cone of $\rho$ in $H^2(Y;\C)$ is denoted $\mathrm{Amp}(\rho)$. By Theorem \ref{thm:main}, we have a projective morphism $\mc{Y} \rightarrow \CM(\Gamma)$ over $\mf{c} \simeq H^2(Y;\C)$. As explained in \cite{Namikawa3}, the set $\mc{D} \subset \mf{c}$ defined in the introduction corresponds to the closed subset $\mc{D}' \subset H^2(Y;\C)$ consisting of all points $t$ such that the fiber $\mc{Y}_t := \boldsymbol{\nu}^{-1}(t)$ is not affine. 

\begin{thm}[Main Theorem, \cite{Namikawa3}]\label{thm:Nam3main}
\begin{enumerate}
\item There are finitely many hyperplanes $\{H_i \}_{i \in I}$ in $H^2(Y;\Q)$ such that $\mc{D}' = \bigcup_{i \in I} (H_i)_{\C}$. 
\item There are only finitely many $\Q$-factorial terminalizations $\{ \rho_k : Y_k \rightarrow V / \Gamma \}_{k \in K}$ of $V / \Gamma$. 
\item The set of open chambers determined by the real hyperplanes $\{ (H_i)_{\R} \}_{i \in I}$ coincides with the set $\{ w(\mathrm{Amp}(\rho_k)) \}$, where $w \in W$ and $k \in K$. 
\end{enumerate}
\end{thm}

For a topological space $X$, we abuse notation and let $\pi_0(X)$ denote the number of connected components of $X$. Let $\mathrm{Mov}(\rho)^{\circ} = \mathrm{Mov}(\rho) \smallsetminus \bigcup_{i \in I} (H_{i})_{\R}$. By Theorem \ref{thm:Nam3main} (3) and \cite[Lemma 6]{Namikawa3}, $\mathrm{Mov}(\rho)^{\circ}$ equals $\bigsqcup_{k \in K} \mathrm{Amp}(\rho_k)$. Hence, the number $|K|$ of pairwise non-isomorphic $\Q$-factorial terminalizations of $V / \Gamma$ equals $\pi_0(\mathrm{Mov}(\rho)^{\circ})$. Proposition 2.19 of \cite{BPWQuant1} implies that 
$$
H^2(Y;\R) \smallsetminus \bigcup_{i \in I} (H_{i})_{\R} = \bigsqcup_{w \in W} w(\mathrm{Mov}(\rho)^{\circ}). 
$$
Thus,   
\begin{equation}\label{eq:pizero}
\pi_0 \left( H^2(Y; \R) \smallsetminus \bigcup_i H_i \right) = |W| \cdot |K|. 
\end{equation}
Zaslavsky's Theorem, \cite[Theorem 2.68]{OrlikTeraoBook}, says that the number of connected components of the complement $H^2(Y; \R) \smallsetminus \bigcup_i (H_i)_{\R}$ to the real hyperplane arrangement $\bigcup_i (H_i)_{\R}$ equals the dimension of the cohomology ring of the complement  $H^2(Y; \C) \smallsetminus \mc{D}'$ to the complexification $\mc{D}'$ of the real hyperplane arrangement. As explained above, Theorem \ref{thm:main} and Theorem \ref{thm:Nam3main} (1) imply that $H^2(Y; \C) \smallsetminus \mc{D}' \simeq \mf{c} \smallsetminus \mc{D}$. Hence 
\begin{equation}\label{eq:pizeroC}
\pi_0 \left( H^2(Y; \R) \smallsetminus \bigcup_i H_i \right) = \dim_{\C} H^*(\mf{c} \smallsetminus \mc{D}; \C). 
\end{equation}
Thus, the claim of Theorem \ref{thm:count} follows from equations (\ref{eq:pizero}) and (\ref{eq:pizeroC}). 

\subsection{The wreath product $\s_n \wr G$.}\label{sec:wreath} In this section we deduce Proposition \ref{prop:countWeyl} from Theorem  \ref{thm:count}. The proposition is trivially true for $n = 1$ due the uniqueness of minimal resolutions. Therefore we assume that $n > 1$. In this case, the Namikawa Weyl group $W$ of $V / \Gamma$ equals $\Z_2 \times W_G$. Let $\h$ be the Cartan algebra on which $W_G$ acts. By \cite[Theorem 1.4]{MarsdenWeinsteinStratification}, there is an isomorphism of vector spaces $\mf{c} \simeq \C \times \h$, lifting to an identification of $\CM(\Gamma)$ with a certain moduli space of representations of a deformed preprojective algebra - see \textit{loc. cit.} for details. The proof of Lemmata 4.3 and 4.4 of \cite{GordonQuiver} shows that the set $\mc{D}$ over which the fibers $\CM_{\bc}(\Gamma)$ are singular is the union of hyperplanes in $\C \times \h$ given by 
$$
H_{\lambda,m} := \{ (\alpha, x) \in \C \times \h \ | \ \lambda(x) + m \alpha = 0 \} \quad \textrm{and} \quad \alpha = 0,
$$
where $\lambda \in R$, the root system of the Weyl group $W_G$, and $1-n \le m \le n-1$. In the language of \cite[Chapter 1]{OrlikTeraoBook}, this hyperplane arrangement is the cone over the affine hyperplane arrangement 
$$
\mc{A}=  \left\{ \overline{H}_{\lambda,m} =  \{ x \in \h \ | \ \lambda(x) + m = 0 \} \ | \  \lambda \in R, \ 1-n \le m \le n-1 \ \right\}. 
$$
Therefore \cite[Proposition 2.51]{OrlikTeraoBook} implies that 
$$
\dim_{\C} H^*(\mf{c} \smallsetminus \mc{D}; \C) = 2 \dim_{\C} H^* \left(\h \smallsetminus \bigcup_{H \in \mc{A}} H; \C\right).
$$ 
Since $|W| = 2 |W_G|$, the result follows from the above equality, together with equation (1) of Theorem 1.2 in \cite{athanasiadis} and the fact that $|W_G| = \prod_{i = 1}^{\ell} (e_i + 1)$.

\subsection{Exceptional groups}\label{sec:except} Other than the infinite series $\s_n \wr G$, there are only two exceptional groups that are known to admit symplectic resolutions. These are $Q_8 \times_{\Z_2} D_8$ and $G_4$, and their explicit descriptions as subgroups of $\mathrm{Sp}(\C^4)$ can be found in \cite{smoothsra} and \cite{Singular} respectively. 

First we consider the group $Q_8 \times_{\Z_2} D_8$. As shown in \cite{smoothsra}, we have $\mf{c} = \C \{ \bc_1, \ds, \bc_5 \}$ and $\mc{B}$ consists of five elements $B_1, \ds, B_5$. Each minimal parabolic $\underline{\Gamma}_{B_i}$ is isomorphic to $\Z_2$ and the corresponding quotients $\Xi(B)$ are always trivial. Thus, $W_{B_i} = \s_2$ for all $i$ and if $a_i$ is the generator of $W_{B_i}$ then $a_i \cdot \bc_j = (-1)^{\delta_{ij}} \bc_j$. There are 21 hyperplanes in $\mf{c}$ given by the 16 of the form $\bc_1 \pm \bc_2  \pm \bc_3 \pm \bc_4 \pm \bc_5 = 0$ and the five of the from $\bc_i = 0$. Using\footnote{A copy of the code used to make this calculation is available from the author.} the computer algebra program $\mathrm{MAGMA}$ \cite{MAGMA}, it is possible to calculate that the Poincar\'e polynomial of the Orlik-Solomon algebra equals $1 + 21 t + 170 t^2 + 650 t^3 + 1125 t^4 + 625 t^5$. This implies that the quotient $V / \Gamma$ admits $81$ distinct symplectic resolutions. 

For the group $G_4$, the proof of \cite[Theorem 1.4]{MoPositiveChar} shows that $\mc{D} = H_1 \cup H_2 \cup H_3$, where
$$
H_1 = \{ \bc_1 + \bc_2 = 0 \}, \quad H_2 = \{ \omega  \bc_1 + \omega ^2 \bc_2 = 0 \}, \quad H_3 = \{ \omega^2 \bc_1 + \omega \bc_2 = 0 \},
$$
with $\omega$ a primitive $3$rd root of unity. Since $\dim H_i \cap H_j = 0$ for $i \neq j$, the only dependent subset of $\{ H_1, H_2, H_3 \}$ is $\{ H_1, H_2, H_3 \}$ itself. Therefore, \cite[Definition 3.5]{OrlikTeraoBook} says that the Orlik-Solomon algebra associated to the arrangement $\mc{D}$ is the quotient of the exterior algebra $\wedge^{\idot}(x_1,x_2,x_3)$ by the two-sided ideal generated by 
$$
\partial (x_1 \wedge x_2 \wedge x_3) = x_2 \wedge x_3 - x_1 \wedge x_3 + x_1 \wedge x_2.
$$
Hence, the Orlik-Solomon Theorem \cite[Theorem 5.90]{OrlikTeraoBook} says that 
$$
H^{\idot}(\mf{c} \smallsetminus \mc{D},\C) \simeq \frac{\wedge^{\idot}(x_1,x_2,x_3)}{\langle x_2 \wedge x_3 - x_1 \wedge x_3 + x_1 \wedge x_2 \rangle}.
$$
The Orlik-Solomon algebra has basis $\{ 1, x_1, x_2, x_3, x_1 \wedge x_3 ,x_1 \wedge x_2 \}$; this can be seen directly, or by applying \cite[Theorem 3.43]{OrlikTeraoBook}. The Weyl group for $G_4$ is $\Z_3$. Hence Theorem \ref{thm:count} implies that there are $2$ non-isomorphic symplectic resolutions of $\C^4 / G_4$. This implies that the two symplectic resolutions constructed in \cite{LehnSorger} exhaust all symplectic resolutions.

\def\cprime{$'$} \def\cprime{$'$} \def\cprime{$'$} \def\cprime{$'$}
  \def\cprime{$'$} \def\cprime{$'$} \def\cprime{$'$} \def\cprime{$'$}
  \def\cprime{$'$} \def\cprime{$'$} \def\cprime{$'$} \def\cprime{$'$}
  \def\cprime{$'$}



\end{document}